\documentclass[11pt,reqno]{amsart}
\usepackage[margin=1in]{geometry}
\usepackage{hyperref, amssymb, amsfonts, amsthm, amsrefs, mathtools, enumerate, verbatim, color, makecell}

\newtheorem{Theorem}{Theorem}[section]

\newtheorem{Lemma}[Theorem]{Lemma}
\newtheorem{Corollary}[Theorem]{Corollary}

\theoremstyle{definition}
\newtheorem{Definition}[Theorem]{Definition}

\newtheorem{Conjecture}[Theorem]{Conjecture}

\newcommand{\rca}{\mathsf{RCA}_0}
\newcommand{\wkl}{\mathsf{WKL}_0}
\newcommand{\aca}{\mathsf{ACA}_0}
\newcommand{\stet}{\mathsf{sTET}_{[0,1]}}
\newcommand{\rec}{\textup{REC}}

\DeclareMathOperator{\wks}{\mathsf{WeakSystem}}
\DeclareMathOperator{\sts}{\mathsf{StrongSystem}}

\DeclareMathOperator{\dom}{\mathrm{dom}}

\newcommand{\andd}{\wedge}
\newcommand{\orr}{\vee}
\newcommand{\la}{\langle}
\newcommand{\ra}{\rangle}
\newcommand{\da}{{\downarrow}}
\newcommand{\ua}{{\uparrow}}
\newcommand{\imp}{\rightarrow}

\newcommand{\biimp}{\leftrightarrow}

\newcommand{\Nb}{\mathbb{N}}
\newcommand{\Qb}{\mathbb{Q}}
\newcommand{\Rb}{\mathbb{R}}

\usepackage{xcolor}	
\usepackage{soul}

\title{The reverse mathematics of the Tietze extension theorem}

\author{Paul Shafer}
\address{Department of Mathematics\\
Ghent University\\
Krijgslaan 281 S22\\
B-9000 Ghent\\
Belgium}
\email{paul.shafer@ugent.be}
\urladdr{http://cage.ugent.be/~pshafer/}
\thanks{Paul Shafer is an FWO Pegasus Long Postdoctoral Fellow.}

\date{\today}

\begin{document}

\begin{abstract}
We prove that several versions of the Tietze extension theorem for functions with moduli of uniform continuity are equivalent to $\wkl$ over $\rca$.  This confirms a conjecture of Giusto and Simpson~\cite{GiustoSimpson} that was also phrased as a question in Montalb\'an's \emph{Open questions in reverse mathematics}~\cite{Montalban}.
\end{abstract}

\maketitle

\section{Introduction}

The \emph{Tietze extension theorem} states that if $X$ is a metric space, $C \subseteq X$ is closed, and $f \colon C \imp \Rb$ is continuous, then there is a continuous $F \colon X \imp \Rb$ extending $f$ (meaning that $F(x) = f(x)$ for all $x \in C$).  It is a fundamental theorem of real analysis and topology, and, as such, the question of its logical strength is natural and ripe for consideration.  In this work, we analyze the logical strengths of formalized versions of the Tietze extension theorem in the setting of \emph{reverse mathematics}, a foundational program designed by Friedman to classify mathematical theorems according to the strengths of the axioms required to prove them~\cite{Friedman}.

In reverse mathematics, we fix a weak base axiom system $\wks$ for second-order arithmetic and consider the implications that are provable in $\wks$.  If $\varphi$ and $\psi$ are two statements in second-order arithmetic, typically expressing two well-known theorems, and $\wks \vdash \varphi \imp \psi$, then we say that $\varphi$ implies $\psi$ over $\wks$ and think of the logical strength of $\varphi$ as being at least that of $\psi$.  We also like to appeal to the equivalence of $\wks \vdash \varphi \imp \psi$ and $\wks + \varphi \vdash \psi$ in order to think of the strength of $\varphi$ in terms of the additional statements $\psi$ that become provable once $\varphi$ is considered as a new axiom and added to the axioms of $\wks$.

Often, as in this work, we wish to compare a theorem $\varphi$ to an axiom system $\sts$ that is stronger than $\wks$ and proves $\varphi$.  In this situation, if $\wks + \varphi \vdash \psi$ for every axiom $\psi$ of $\sts$, then we say that $\varphi$ is equivalent to $\sts$ over $\wks$.  The proof of (the axioms of) $\sts$ from $\wks + \varphi$ is called a \emph{reversal}, from which `reverse mathematics' gets its name.  It is a remarkable phenomenon that equivalences of this sort are the usual case:  a theorem is typically either provable in the standard $\wks$ or equivalent to one of four well-known stronger systems.  These five systems together are known as the Big Five.  There are, however, many fascinating examples of misfit theorems as well, and we refer the reader to~\cite{Hirschfeldt} for the tip of that particular iceberg.

It is possible to formalize the Tietze extension theorem in second-order arithmetic in several different ways, and different formalizations may exhibit different logical strengths.  The logical systems in play are the first three of the Big Five, which are
\begin{itemize}
\item the base system $\rca$ (for \emph{recursive comprehension axiom}), which corresponds to computable mathematics and is the standard $\wks$;

\item the stronger system $\wkl$ (for \emph{weak K\"onig's lemma}), which adds the ability to make compactness arguments; and

\item the yet stronger system $\aca$ (for \emph{arithmetical comprehension axiom}), which adds the ability to form sets defined by any number of first-order quantifiers (but no second-order quantifiers).
\end{itemize}

The differences among the formalizations of the Tietze extension theorem that we consider arise from two sources.  The first source is the problem of coding closed subsets of complete separable metric spaces, which are most naturally thought of as third-order objects, as second-order objects.  Closed sets can be coded by negative information (in which case they are simply called \emph{closed}), positive information (in which case they are called \emph{separably closed}), or both simultaneously (in which case they are called \emph{closed and separably closed}).  In a compact complete separable metric space, a set is closed if and only if it is separably closed, but both directions of this equivalence are themselves equivalent to $\aca$ over $\rca$~\cite[Theorem~3.3]{Brown}.  These notions of closedness are thus distinct when working in $\rca$.

The second source of differences is the fact that the statement ``every continuous function $f \colon X \imp \Rb$ on a compact complete separable metric space $X$ is uniformly continuous'' is equivalent to $\wkl$ over $\rca$ (see \cite[Theorem~IV.2.2 and Theorem~IV.2.3]{SimpsonSOSOA}) and therefore has non-trivial logical strength.  Here `uniformly continuous' means having a \emph{modulus of uniform continuity}, which is a function that, when given an $\epsilon > 0$, returns a $\delta > 0$ such that $(\forall x, y \in X)(d(x,y) < \delta \imp d(f(x),f(y)) < \epsilon)$.  Thus though the two statements
\smallskip
\begin{itemize}
\item[(1)] For every compact complete separable metric space $X$, every closed $C \subseteq X$, and every continuous $f \colon C \imp \Rb$, there is a continuous $F \colon X \imp \Rb$ extending $f$.

\smallskip

\item[(2)] For every compact complete separable metric space $X$, every closed $C \subseteq X$, and every uniformly continuous $f \colon C \imp \Rb$, there is a uniformly continuous $F \colon X \imp \Rb$ extending $f$.
\end{itemize}
\smallskip
are obviously equivalent in ordinary mathematics, the situation over $\rca$ is more complicated.  Following Giusto and Simpson's terminology from~\cite{GiustoSimpson}, we call statement (1) the \emph{Tietze extension theorem} and statement (2) the \emph{strong Tietze extension theorem}.  The following list summarizes some of the known results.
\begin{itemize}
\item The Tietze extension theorem for closed sets (i.e., the negative information coding) is provable in $\rca$ (see \cite[Theorem~II.7.5]{SimpsonSOSOA}).  In this case the assumption that $X$ is compact may be dropped if $f$ is assumed to be bounded.

\item The Tietze extension theorem for separably closed sets is equivalent to $\aca$ over $\rca$~\cite[Theorem~6.9]{GiustoSimpson}.

\item The strong Tietze extension theorem for separably closed sets is equivalent to $\wkl$ over $\rca$~\cite[Theorem~6.14]{GiustoSimpson}.

\item The strong Tietze extension theorem for closed sets is provable in $\wkl$ because the Tietze extension theorem for closed sets is provable in $\rca$, and $\wkl$ proves that continuous functions on compact complete separable metric spaces are uniformly continuous.

\item The strong Tietze extension theorem for closed and separably closed sets is not provable in $\rca$.  In fact, it implies the existence of diagonally non-recursive functions~\cite[Lemma~6.17]{GiustoSimpson}.
\end{itemize}

Notice that the above list of results leaves open the precise logical strength of the strong Tietze extension theorem for closed sets and for closed and separably closed sets.  Giusto and Simpson conjecture that these theorems are equivalent to $\wkl$.  Specifically, they make the following conjecture. 

\begin{Conjecture}[{\cite[Conjecture~6.15]{GiustoSimpson}}]\label{conj-sTET}
The following are equivalent over $\rca$.
\begin{itemize}
\item[(1)] $\wkl$.

\item[(2)] Let $X$ be a compact complete separable metric space, let $C$ be a closed subset of $X$, and let $f \colon C \imp \Rb$ be a continuous function with a modulus of uniform continuity.  Then there is a continuous function $F \colon X \imp \Rb$ with a modulus of uniform continuity that extends $f$.

\item[(3)] Same as (2) with `closed' replaced by `closed and separably closed.'

\item[(4)] Special case of (2) with $X = [0,1]$.

\item[(5)] Special case of (3) with $X = [0,1]$.
\end{itemize}
\end{Conjecture}

The question of whether or not this conjecture holds also appears as Question~16 in Montalb\'an's \emph{Open questions in reverse mathematics}~\cite{Montalban}.  Let $\stet$ denote statement~(5) in Conjecture~\ref{conj-sTET} (the notation is chosen to evoke \emph{the strong Tietze extension theorem for $[0,1]$}).  We prove that Conjecture~\ref{conj-sTET} is true by proving that $\rca + \stet \vdash \wkl$.

Before continuing, we remark that Giusto and Simpson's \emph{Located sets and reverse mathematics}~\cite{GiustoSimpson}, which contains Conjecture~\ref{conj-sTET}, is largely concerned with the notion of a located set, where a closed or separably closed subset $C$ of a complete separable metric space $X$ is called \emph{located} if there is a continuous distance function $f \colon X \imp \Rb$, where $f(x) = \inf\{d(x,y) : y \in C\}$ for every $x \in X$.  With the assumption of locatedness, the equivalence between closed and separably closed becomes provable in $\rca$:  for compact complete separable metric spaces, $\rca$ proves that every closed and located set is separably closed and that every separably closed and located set is closed.  Furthermore, the strong Tietze extension theorem for closed and located sets (and thus for separably closed and located sets) is provable in $\rca$.  These results and many others appear in~\cite{GiustoSimpson}.  However, located sets are not relevant to Conjecture~\ref{conj-sTET}, so we make no use of them here.

For ease of comparison, the table below displays the strengths of eight versions of the Tietze extension theorem, taking into account the confirmation of Conjecture~\ref{conj-sTET} proven here.  The row labeled `Tietze extension theorem' corresponds to versions of the theorem where $f$ and its extension are not required to be uniformly continuous, and the row labeled `strong Tietze extension theorem' corresponds to versions of the theorem where $f$ and its extension are required to be uniformly continuous.  The columns represent the different assumptions on the domain $C$ of $f$.  The column labeled `located' means that $C$ is assumed to be closed and located, which in $\rca$ is equivalent to assuming that $C$ is separably closed and located.

\begin{center}
\begin{tabular}{|l|c|c|c|c|}
\hline
	& \makecell{located}
	& \makecell{closed \& separably closed}
	& \makecell{closed}
	& \makecell{separably closed}\\
\hline
\makecell{Tietze extension theorem}
	& \makecell{$\rca$}
	& \makecell{$\rca$}
	& \makecell{$\rca$}
	& \makecell{$\aca$}\\
\hline
\makecell{strong Tietze extension theorem}
	& \makecell{$\rca$}
	& \makecell{$\wkl$}
	& \makecell{$\wkl$}
	& \makecell{$\wkl$}\\
\hline
\end{tabular}
\end{center}

\section{Background}

We introduce the systems $\rca$, $\wkl$, and $\aca$ and then define in $\rca$ the analytic and topological notions relevant to Conjecture~\ref{conj-sTET}.  The standard reference for reverse mathematics is Simpson's \emph{Subsystems of Second Order Arithmetic}~\cite{SimpsonSOSOA}, and almost all of this section's material can be found in expert detail therein.  Simpson's book also contains many, many examples of theorems that are provable in $\rca$, theorems that are equivalent to $\wkl$ over $\rca$, and theorems that are equivalent to $\aca$ over $\rca$.

\subsection{$\rca$, $\wkl$, and $\aca$}

The axioms of $\rca$ are:  a first-order sentence expressing that $\Nb$ is a discretely ordered commutative semi-ring with identity; the \emph{$\Sigma^0_1$ induction scheme}, which consists of the universal closures (by both first- and second-order quantifiers) of all formulas of the form
\begin{align*}
[\varphi(0) \andd \forall n(\varphi(n) \imp \varphi(n+1))] \imp \forall n \varphi(n),
\end{align*}
where $\varphi$ is $\Sigma^0_1$; and the \emph{$\Delta^0_1$ comprehension scheme}, which consists of the universal closures (by both first- and second-order quantifiers) of all formulas of the form
\begin{align*}
\forall n (\varphi(n) \biimp \psi(n)) \imp \exists X \forall n(n \in X \biimp \varphi(n)),
\end{align*}
where $\varphi$ is $\Sigma^0_1$, $\psi$ is $\Pi^0_1$, and $X$ is not free in $\varphi$.

$\rca$ is the standard base system and captures what might be called \emph{effective mathematics}.  The name `$\rca$,' which stands for \emph{recursive comprehension axiom}, refers to the $\Delta^0_1$ comprehension scheme because a set $X$ is $\Delta^0_1$ in a set $Y$ if and only if $X$ is recursive in $Y$.  The subscript `$0$' refers to the fact that induction in $\rca$ is limited to $\Sigma^0_1$ formulas.

$\rca$ proves enough number-theoretic facts to implement the codings of finite sets and sequences that are ubiquitous in recursion theory.  Therefore, in $\rca$ we can represent the set $\Nb^{<\Nb}$ of all finite sequences as well as its subset $2^{<\Nb}$ of all finite binary sequences, and we can give the usual definition of a tree as subset of $\Nb^{<\Nb}$ that is closed under initial segments.  Thus, in $\rca$ we can formulate (but not prove) \emph{weak K\"onig's lemma}, which is the statement ``every infinite subtree of $2^{<\Nb}$ has an infinite path.''  $\wkl$ is then the system $\rca + \text{weak K\"onig's lemma}$.  The fact that there is a recursive infinite subtree of $2^{<\Nb}$ with no recursive infinite path can be used to show that $\wkl$ is strictly stronger than $\rca$.  $\wkl$ captures the mathematics of compactness.  For example, $\wkl$ is equivalent to the Heine-Borel compactness of $[0,1]$ (see~\cite[Theorem~IV.1.2]{SimpsonSOSOA}), a fact that is crucial for our analysis of $\stet$.

An important strategy for proving that a theorem implies $\wkl$ over $\rca$ is to employ the following lemma, which states that $\wkl$ is equivalent over $\rca$ to the statement that for every pair of injections with disjoint ranges, there is a set that separates the two ranges.

\begin{Lemma}[see {\cite[Lemma~IV.4.4]{SimpsonSOSOA}}]\label{lem-sep}
The following are equivalent over $\rca$.
\begin{itemize}
\item[(i)] $\wkl$.
\item[(ii)] If $g_0, g_1 \colon \Nb \imp \Nb$ are injections such that $\forall m \forall n(g_0(m) \neq g_1(n))$, then there is a set $X$ such that $\forall m(g_0(m) \in X \andd g_1(m) \notin X)$.
\end{itemize}
\end{Lemma}

For comparison, $\aca$, introduced next, is equivalent over $\rca$ to the statement that for every injection there is a set consisting of exactly the elements in the injection's range.

The axioms of $\aca$ are those of $\rca$, plus the \emph{arithmetical comprehension scheme}, which consists of the universal closures (by both first- and second-order quantifiers) of all formulas of the form
\begin{align*}
\exists X \forall n(n \in X \biimp \varphi(n)),
\end{align*}
where $\varphi$ is an arithmetical formula in which $X$ is not free.

Jockusch and Soare's famous low basis theorem~\cite{JockuschSoare} can be used to prove that $\aca$ is strictly stronger than $\wkl$.  The strength of $\aca$ is great enough to provide a natural and extensive development of most classical mathematics.  Though we do not make further use of $\aca$ here, it is relevant to the discussion in the introduction.

\subsection{Analytic and topological notions in $\rca$}

Following~\cite[Section~II.4]{SimpsonSOSOA}, we code integers as pairs of natural numbers and rational numbers as pairs of integers.  A real number is then coded by a sequence of rational numbers $\la q_k : k \in \Nb \ra$ such that $\forall k \forall i(|q_k - q_{k+i}| \leq 2^{-k})$.  The expression `$x \in \Rb$' abbreviates the predicate ``$x$ codes a real number.''  Two real numbers coded by $\la q_k : k \in \Nb \ra$ and $\la q_k' : k \in \Nb \ra$ are equal if $\forall k(|q_k - q_k'| \leq 2^{-k+1})$.

The definition of a complete separable metric space generalizes the coding of reals by rapidly converging Cauchy sequences.

\begin{Definition}[$\rca$; {see \cite[Definition~II.5.1]{SimpsonSOSOA}}]
A \emph{complete separable metric space} $\widehat A$ is coded by a non-empty set $A \subseteq \Nb$ and a distance function $d \colon A \times A \imp \Rb$ such that, for all $a,b,c \in A$, $d(a,a) = 0$, $d(a,b) = d(b,a) \geq 0$, and $d(a,b) + d(b,c) \geq d(a,c)$.

A \emph{point} in $\widehat A$ is coded by a sequence $\la a_k : k \in \Nb \ra$ of elements of $A$ such that $\forall k \forall i(d(a_k, a_{k+i}) \leq 2^{-k})$.  The expression `$x \in \widehat A$' abbreviates the predicate ``$x$ codes a point in $\widehat A$.''

If $x = \la a_k : k \in \Nb \ra$ and $y = \la b_k : k \in \Nb \ra$ code points in $\widehat A$, then $d(x,y)$ is defined to be $\lim_k d(x_k, y_k)$, and (the points coded by) $x$ and $y$ are defined to be equal if $d(x,y) = 0$.
\end{Definition}

\begin{Definition}[$\rca$; {see~\cite[Definition~III.2.3]{SimpsonSOSOA}}]\label{def-compact}
A complete separable metric space $\widehat A$ is \emph{compact} if there is a sequence of finite sequences $\la \la x_{i,j} : j < n_i \ra : i \in \Nb \ra$ of points in $\widehat A$ such that
\begin{align*}
(\forall z \in \widehat A)(\forall i \in \Nb)(\exists j < n_i)(d(z, x_{i,j}) < 2^{-i}).
\end{align*}
\end{Definition}

For example, the unit interval $[0,1]$ is the complete separable metric space coded by $\{q \in \Qb : 0 \leq q \leq 1\}$ with the usual metric, and the sequence $\la \la j2^{-i} : j \leq 2^i \ra : i \in \Nb \ra$ witnesses that $[0,1]$ is compact.

In complete separable metric spaces, open sets are coded by enumerations of open balls, and closed sets are complements of open sets.  For the purposes of the remaining definitions, $\Qb^+ = \{q \in \Qb : q > 0\}$ denotes the set of positive rationals.

\begin{Definition}[$\rca$; {see~\cite[Definition~II.5.6 and Definition~II.5.12]{SimpsonSOSOA}}]
An \emph{open set} in a complete separable metric space $\widehat A$ is coded by a set $U \subseteq \Nb \times A \times \Qb^+$.  A point $x \in \widehat A$ belongs to $U$ (abbreviated `$x \in U$') if
\begin{align*}
\exists n \exists a \exists r(d(x,a) < r \andd \la n, a, r \ra \in U).
\end{align*}

A \emph{closed set} $C$ in a complete separable metric space is also coded by a set $U \subseteq \Nb \times A \times \Qb^+$, but now a point $x \in \widehat A$ belongs to $C$ (abbreviated `$x \in C$') if $x \notin U$.
\end{Definition}

The idea here is that the pair $\la a, q \ra \in A \times \Qb^+$ codes the open ball $B(a,q)$ of radius $q$ centered at $a$ and that a set $U \subseteq \Nb \times A \times \Qb^+$ codes some sequence $\la B(a_k, q_k) : k \in \Nb \ra$ of open balls and hence codes the open set $\bigcup_{k \in \Nb}B(a_k, q_k)$.  Thus in this scheme, open sets are coded by positive information (enumerations of open balls contained in the open set), and closed sets are coded by negative information (enumerations of open balls disjoint from the closed set).  Alternatively, a closed set $C$ can be coded by positive information by enumerating a sequence of points whose closure is $C$.  Such a set is called \emph{separably closed}.

\begin{Definition}[$\rca$; {see~\cite[Definition~4.1]{GiustoSimpson}}]
A \emph{separably closed set} in a complete separable metric space $\widehat A$ is coded by a sequence $C = \la x_k : k \in \Nb \ra$ of points in $\widehat A$.  A point $x \in \widehat A$ belongs to $C$ (abbreviated `$x \in C$') if
\begin{align*}
(\forall q \in \Qb^+)(\exists n \in \Nb)(d(x, x_n) < q).
\end{align*}
\end{Definition}

We can now define a complete separable metric space to be \emph{Heine-Borel compact} if for every sequence $\la U_k : k \in \Nb \ra$ of open sets such that $(\forall x \in \widehat A)(\exists k \in \Nb)(x \in U_k)$, there is an $N \in \Nb$ such that $(\forall x \in \widehat A)(\exists k < N)(x \in U_k)$.  Although $\rca$ proves that $[0,1]$ is a compact complete separable metric space in the sense of Definition~\ref{def-compact}, the Heine-Borel compactness of $[0,1]$ is equivalent to $\wkl$ over $\rca$.

\begin{Theorem}[{see~\cite[Theorem~IV.1.2 and Theorem~IV.1.5]{SimpsonSOSOA}}]\label{thm-WKLCompact}
The following are equivalent over $\rca$.
\begin{itemize}
\item[(1)] $\wkl$.
\item[(2)] Every compact complete separable metric space is Heine-Borel compact.
\item[(3)] The unit interval $[0,1]$ is Heine-Borel compact.
\item[(4)] For every sequence $\la (a_k, b_k) : k \in \Nb \ra$ of intervals with rational endpoints such that $(\forall x \in [0,1])(\exists k \in \Nb)(x \in (a_k, b_k))$, there is an $N \in \Nb$ such that $(\forall x \in [0,1])(\exists k < N)(x \in (a_k, b_k))$.
\end{itemize}
\end{Theorem}

Finally, we define continuous functions and moduli of uniform continuity.

\begin{Definition}[$\rca$; {see~\cite[Definition~II.6.1]{SimpsonSOSOA}}]\label{def-continuous}
Let $\widehat A$ and $\widehat B$ be complete separable metric spaces.  A continuous partial function from $\widehat A$ to $\widehat B$ is coded by a set $\Phi \subseteq \Nb \times A \times \Qb^+ \times B \times \Qb^+$ that satisfies the properties below.  Let $\la a, r \ra \Phi \la b, s \ra$ denote $\exists n(\la n, a, r, b, s \ra \in \Phi)$.  For $a, a' \in A$ and $r, r' \in \Qb^+$, let $\la a', r' \ra \prec \la a, r \ra$ denote $d(a,a') + r' < r$ and similarly for $b, b' \in B$ and $s, s' \in \Qb^+$.  The properties that $\Phi$ must satisfy are that, for all $a, a' \in A$, all $b, b' \in B$, and all $r, r', s, s' \in \Qb^+$, 
\begin{itemize}
\item if $\la a, r \ra \Phi \la b, s \ra$ and $\la a, r \ra \Phi \la b', s' \ra$, then $d(b, b') \leq s + s'$;

\item if $\la a, r \ra \Phi \la b, s \ra$ and $\la a', r' \ra \prec \la a, r \ra$, then $\la a', r' \ra \Phi \la b, s \ra$; and

\item if $\la a, r \ra \Phi \la b, s \ra$ and $\la b, s \ra \prec \la b', s' \ra$, then $\la a, r \ra \Phi \la b', s' \ra$.
\end{itemize}
\end{Definition}

The domain of the function $f$ coded by $\Phi$ is the set of all $x \in \widehat A$ such that
\begin{align*}
(\forall q \in \Qb^+)(\exists \la a, r \ra \in A \times \Qb^+)(\exists \la b, s \ra \in B \times \Qb^+)(\la a, r \ra \Phi \la b, s \ra \andd d(x,a) < r \andd s < q).
\end{align*}
If $x \in \dom f$, then $f(x)$ is the unique $y \in \widehat B$ such that
\begin{align*}
(\forall \la a, r \ra \in A \times \Qb^+)(\forall \la b, s \ra \in B \times \Qb^+)((d(x,a) < r \andd \la a, r \ra \Phi \la b, s \ra) \imp d(y,b) \leq s).
\end{align*}

The idea behind Definition~\ref{def-continuous} is that $\Phi$ enumerates pairs of open balls $\la B(a,r), B(b,s) \ra$ (i.e., the pairs of balls coded by the $\la a, r \ra$ and $\la b, s \ra$ such that $\la a, r \ra \Phi \la b, s \ra$) with the property that if $f$ is the function being coded by $\Phi$ and $x$ is in both $B(a,r)$ and $\dom f$, then $f(x)$ is in the closure of $B(b,s)$.

\begin{Definition}[$\rca$; {see~\cite[Definition~IV.2.1]{SimpsonSOSOA}}]
Let $\widehat A$ and $\widehat B$ be complete separable metric spaces, and let $f$ be a partial continuous function from $\widehat A$ to $\widehat B$.  A \emph{modulus of uniform continuity} for $f$ is a function $h \colon \Nb \imp \Nb$ such that
\begin{align*}
(\forall x, y \in \dom f)(\forall n \in \Nb)(d(x,y) < 2^{-h(n)} \imp d(f(x),f(y)) < 2^{-n}).
\end{align*}
\end{Definition}

\section{Reversing the strong Tietze extension theorem to weak K\"onig's lemma}

In their analysis of $\stet$, Giusto and Simpson first show that $\rca \nvdash \stet$ by showing that $\stet$ fails in $\rec$, the model of $\rca$ whose first-order part is the standard natural numbers and whose second-order part is the recursive sets~\cite[Lemma~6.16]{GiustoSimpson}.  To do this, they take advantage of Theorem~\ref{thm-WKLCompact}, the fact that $\wkl$ fails in $\rec$, and the fact that $\rca$ proves that a continuous real-valued function on $[0,1]$ has a modulus of uniform continuity if and only if it has a \emph{Weierstra{\ss} approximation} (see~\cite[Theorem~IV.2.4]{SimpsonSOSOA}).  Here, a Weierstra{\ss} approximation of a continuous function $f \colon [0,1] \imp \Rb$ is a sequence of polynomials $\la p_n : n \in \Nb \ra$ from $\Qb[x]$ such that $(\forall n \in \Nb)(\forall x \in [0,1])(|f(x) - p_n(x)| < 2^{-n})$.

The goal in proving that $\rec \not\models \stet$ is thus to produce a recursive code for a closed and separably closed $C \subseteq [0,1]$, a recursive code for a continuous $f \colon C \imp \Rb$, and a recursive modulus of uniform continuity for $f$ such that no continuous extension of $f$ to $[0,1]$ has a recursive Weierstra{\ss} approximation.  To this end, let $I_e = [2^{-(2e+1)}, 2^{-2e}]$ for each $e \in \Nb$, and let $D = \{0\} \cup  \bigcup_{e \in \Nb}I_e$.  We call $D$ the \emph{pre-domain} of $f$, as $C \subseteq D$ is obtained from $D$ by enumerating additional open intervals into the complement of $C$.  The plan is to define $f(0) = 0$, then for each $e \in \Nb$ to define $C$ and $f$ on $I_e$ to diagonalize against $\Phi_e$ being a Weierstra{\ss} approximation to an extension of $f$.  Thus on each $I_e$ we implement the following strategy.  First, by the fact that Theorem~\ref{thm-WKLCompact} item~(4) fails in $\rec$, fix a recursive enumeration $\la (a_k, b_k) : k \in \Nb \ra$ of open intervals with rational endpoints that covers (the recursive reals in) $[0,1]$ but has no finite subcover.  Transfer this cover to a cover $\la (a^e_k, b^e_k) : k \in \Nb \ra$ of $I_e$ that has no finite subcover by the linear transformation $x \mapsto (x+1)/2^{2e+1}$.  Enumerate the intervals of $\la (a^e_k, b^e_k) : k \in \Nb \ra$ into the complement of $C$ until a stage $s$ is reached that witnesses $\Phi_{e,s}(2e+1)\da = p$, where $p$ is (a code for) a polynomial in $\Qb[x]$.  If $\Phi_e(2e+1)\ua$, then $s$ is never found, and all the intervals in the sequence $\la (a^e_k, b^e_k) : k \in \Nb \ra$ are enumerated into the complement of $C$.  In this case, $I_e$ is erased from the domain of $f$, so we do not need to take any action to define $f$ there.  If $s$ is found, then at stage $s$ only the intervals of $\la (a^e_k, b^e_k) : k < s \ra$ have been enumerated into the complement of $C$.  We then stop the enumeration, which makes $C \cap I_e = I_e \setminus \bigcup_{k < s}(a^e_k, b^e_k)$.  As no finite set of intervals from $\la (a^e_k, b^e_k) : k \in \Nb \ra$ covers $I_e$, we can find a rational $q \in I_e \setminus \bigcup_{k < s}(a^e_k, b^e_k)$.  Then we define $f$ on $I_e \setminus \bigcup_{k < s}(a^e_k, b^e_k)$ by making it be constantly $2^{-2e}$ if $p(q) \leq 0$ and making it be constantly $-2^{-2e}$ otherwise.  In both cases we ensure that $|f(q) - p(q)| \geq 2^{-2e}$, which successfully diagonalizes against $\Phi_e$ because if $\Phi_e$ were a Weierstra{\ss} approximation to an extension of $f$, then we would have $|f(q) - p(q)| < 2^{-(2e+1)}$.  Furthermore, $C$ is closed and separably closed by Lemma~\ref{lem-SepClosed} below, and it is easy to write down a modulus of uniform continuity for $f$.

Our plan to prove that $\rca + \stet \vdash \wkl$ is to formalize and elaborate upon the preceding argument.  Observe, however, that the above argument relies very heavily on the fact that $[0,1]$ is not Heine-Borel compact in $\rec$.  To replicate this style of argument, we appeal to Theorem~\ref{thm-WKLCompact} and work in $\rca + \neg\wkl$.  The overall strategy is thus to produce the contradiction $\rca + \neg\wkl + \stet \vdash \wkl$.

Let $g_0, g_1 \colon \Nb \imp \Nb$ be two injections with disjoint ranges.  By Lemma~\ref{lem-sep}, we wish to separate the ranges of $g_0$ and $g_1$ using $\stet$.  A first idea would be to follow the proof that $\rec \not\models \stet$ and use $I_e$ to code whether or not $e$ should be in a separating set.  Enumerate the intervals of $\la (a^e_k, b^e_k) : k \in \Nb \ra$ into the complement of $C$ until a stage $s$ is reached that witnesses either $g_0(s) = e$ or $g_1(s) = e$.  If $g_0(s) = e$, then define $f$ to be $2^{-2e}$ on the remaining portion of $I_e$; and if $g_1(s) = e$, then define $f$ to be $-2^{-2e}$ on the remaining portion of $I_e$.  The idea would then be to decode a separating set from an extension $F$ of $f$ by checking whether or not $F$ is $\geq 0$ on $I_e$.  The problem is of course that not every $F(q)$ for $q \in I_e$ correctly codes whether or not $e$ should be in a separating set.  We would need to find a $q \in I_e$ that is sufficiently close to a member of $C$, where the meaning of `sufficiently close' is determined by $F$'s modulus of uniform continuity.

We refine this idea by replacing each $I_e$ with a sequence of disjoint closed intervals $\la I_{e,m} : m \in \Nb \ra$ where the length of each $I_{e,m}$ is at most $2^{-m}$, and we choose a rational $q_{e,m} \in I_{e,m}$ for each $e, m \in \Nb$.  The pre-domain for our $f$ is $\{0\} \cup \bigcup_{e,m \in \Nb}I_{e,m}$.  The refined strategy is to implement the above na\"ive coding plan for $I_e$ on each interval $I_{e,m}$.  In the end, if $e$ is in the range of $g_0$ or $g_1$, then $C \cap I_{e,m}$ is non-empty for every $m \in \Nb$.  So in this case, for every $m \in \Nb$, $q_{e,m}$ is a point in $I_{e,m}$ that is within $2^{-m}$ of a point in $C$.  Thus we are able to decode whether or not $e$ should be in a separating set from an extension of $f$ and the extension's modulus of uniform continuity.

The first lemma says that the closed sets we consider are also separably closed.  It is implicit in~\cite{GiustoSimpson}, but we make it explicit as a matter of convenience.

\begin{Lemma}[$\rca$]\label{lem-SepClosed}
If $\la J_e : e \in \Nb \ra$ is a sequence of pairwise disjoint closed sub-intervals of $[0,1]$ with rational endpoints such that $C = \{0\} \cup \bigcup_{e \in \Nb}J_e$ is closed, then $C$ is also separably closed.
\end{Lemma}

\begin{proof}
Let $Q = \la q_n : n \in \Nb \ra$ be an enumeration of the rationals in $\{0\} \cup \bigcup_{e \in \Nb}J_e$.  We show that the closure of $Q$ is $C$.  Clearly $0$ is in the closure of $Q$, and if $x \in J_e$ it is easy to see that $x$ is in the closure of the rationals in $J_e$.  Conversely, suppose that $x \notin \{0\} \cup \bigcup_{e \in \Nb}J_e$.  Then $x$ is in some open interval $(a,b)$ contained in the complement of $C$.  By shrinking this interval, we can find an $m \in \Nb \setminus \{0\}$ such that $(x-1/m, x+1/m)$ is contained in the complement of $C$.  Thus $\forall n (|x - q_n| \geq 1/m)$, so $x$ is not in the closure of $Q$.
\end{proof}
We remark that in Lemma~\ref{lem-SepClosed}, $Q$ can even be taken to be a set of rationals, rather than a sequence of rationals.  Let $Q$ contain $0$ and the set of rationals $q$ such that there is an $e$ less than $q$'s code with $q \in J_e$.

The next lemma prepares $f$'s pre-domain $\{0\} \cup \bigcup_{e,m \in \Nb}I_{e,m}$.

\begin{Lemma}[$\rca + \neg\wkl$]\label{lem-DomainHelper}
For each $e \in \Nb$, let $I_e = [2^{-(2e+1)}, 2^{-2e}]$.  Then there are pairwise disjoint closed intervals with rational endpoints $\la I_{e,m} : e,m \in \Nb \ra$, rationals $\la q_{e,m} : e,m \in \Nb \ra$, and open intervals with rational endpoints $\la (a^{e,m}_k, b^{e,m}_k) : e,m,k \in \Nb \ra$ such that
\begin{itemize}
\item[(i)] $\{0\} \cup \bigcup_{e,m \in \Nb}I_{e,m}$ is closed;
\item[(ii)] $q_{e,m} \in I_{e,m}$;
\item[(iii)] $I_{e,m} \subseteq I_e$, and the length of $I_{e,m}$ is less than $2^{-m}$;
\item[(iv)] $\la (a^{e,m}_k, b^{e,m}_k) : k \in \Nb \ra$ is an open cover of $I_{e,m}$ with no finite subcover;
\item[(v)] if $\la e,m \ra \neq \la e',m' \ra$, then $I_{e,m}$ and $(a^{e',m'}_k, b^{e',m'}_k)$ are disjoint.
\end{itemize}
\end{Lemma}

\begin{proof}
By $\neg\wkl$ in the form of the negation of Theorem~\ref{thm-WKLCompact} item~(4), fix an open cover $\la (a_k, b_k) : k \in \Nb \ra$ of $[0,1]$ by open intervals with rational endpoints that has no finite subcover.  By adjusting the endpoints of the intervals as necessary, assume that $(\forall k \in \Nb)(-2^{-2} < a_k < b_k < 1+2^{-2})$.  For each $e \in \Nb$, transfer $\la (a_k, b_k) : k \in \Nb \ra$ to $I_e$ via the linear transformation $x \mapsto (x+1)/2^{2e+1}$, and denote the transferred sequence of intervals by $\la (a^e_k, b^e_k) : k \in \Nb \ra$.  Notice that if $e \neq e'$ then $(a^e_k, b^e_k)$ and $(a^{e'}_{k'}, b^{e'}_{k'})$ are disjoint for all $k$ and $k'$.

The procedure described below is clearly uniform in $e$, so we think of fixing an $e \in \Nb$ and enumerating
\begin{itemize}
\item $\la I_{e,m} : m \in \Nb \ra$;
\item $\la q_{e,m} : m \in \Nb \ra$;
\item helper pairwise disjoint open intervals with rational endpoints $\la (c^{e,m}, d^{e,m}) : m \in \Nb \ra$ (used to later define $\la (a^{e,m}_k, b^{e,m}_k) : k \in \Nb \ra$ for each $m$); and
\item an increasing sequence of indices $\la k_{e,m} : m \in \Nb \ra$ and a sequence of finite sets of open intervals with rational endpoints $\la U_{e,m} : m \in \Nb \ra$ such that, for all $m \in \Nb$, $\bigcup_{n \leq m}I_{e,n} \subseteq \bigcup_{k \leq k_{e,m}}(a^e_k, b^e_k)$ and $\bigcup_{k \leq k_{e,m}}(a^e_k, b^e_k) \setminus \bigcup_{n \leq m}I_{e,n} = \bigcup_{n \leq m}\bigcup_{O \in U_{e,n}}O$.
\end{itemize}

To start, let $k_{e,0}$ be the least $k$ such that $(a^e_k, b^e_k)$ intersects $I_e$, and choose an open interval $(c^{e,0}, d^{e,0}) \subseteq I_e \cap (a^e_{k_{e,0}}, b^e_{k_{e,0}})$.  Now choose a closed interval $I_{e,0} \subseteq (c^{e,0}, d^{e,0})$ of length at most $2^{-0}$ and a rational $q_{e,0} \in I_{e,0}$.  Enumerate (the at most finitely many intervals coding) $\bigcup_{k \leq k_{e,0}}(a^e_k, b^e_k) \setminus I_{e,0}$ into $U_{e,0}$.  The construction proceeds in this manner.  Suppose at stage $m+1$ we have $\la I_{e,n} : n \leq m \ra$, $\la q_{e,n} : n \leq m \ra$, $\la (c^{e,n}, d^{e,n}) : n \leq m \ra$, $\la k_{e,n} : n \leq m \ra$, and $\la U_{e,n} : n \leq m \ra$.  The set $D_{e,m} = I_e \setminus \bigcup_{k \leq k_{e,m}}(a^e_k, b^e_k)$ must be a finite union of closed intervals where at least one of the intervals is non-degenerate (i.e., not a point) because no finite collection of intervals from $\la (a^e_k, b^e_k) : k \in \Nb \ra$ covers $I_e$.  Let $k_{e, m+1}$ be the least $k$ such that $(a^e_k, b^e_k)$ intersects a non-degenerate component interval of $D_{e,m}$, and choose an open interval $(c^{e,m+1}, d^{e,m+1}) \subseteq D_{e,m} \cap (a^e_{k_{e, m+1}}, b^e_{k_{e, m+1}})$, a closed interval $I_{e,m+1} \subseteq (c^{e,m+1}, d^{e,m+1})$ of length at most $2^{-(m+1)}$, and a rational $q_{e,m+1} \in I_{e,m+1}$.  Enumerate (the at most finitely many intervals coding) $\bigcup_{k \leq k_{e, m+1}}(a^e_k, b^e_k) \setminus \bigcup_{n \leq m+1} I_{e,n}$ into $U_{e,m+1}$.

Immediately we see that (ii)~and~(iii) are satisfied.  For~(i), consider the closed set $C$ described by the simultaneous enumeration of the open intervals $\la (2^{-(2e+2)}, 2^{-(2e+1)}) : e \in \Nb \ra$ and the open intervals in $\bigcup_{e,m \in \Nb}U_{e,m}$.  Suppose that $x \in \{0\} \cup \bigcup_{e,m \in \Nb}I_{e,m}$.  If $x = 0$, then it is clear that $x \in C$.  If $x \in I_{e,m}$, then $x$ is in no interval of the form $(2^{-(2e+2)}, 2^{-(2e+1)})$, and it is in no interval $O \in \bigcup_{n \in \Nb}U_{e',n}$ for an $e' \neq e$.  Furthermore, $x$ is in no interval $O \in \bigcup_{n \in \Nb}U_{e,n}$ either.  This is because when $I_{e,m}$ is defined at stage $m$ for $e$, $I_{e,m}$ is chosen disjoint from the intervals in $\bigcup_{n<m}U_{e,n}$, and at stages $n \geq m$ the intervals added to $U_{e,n}$ are chosen to be disjoint from $I_{e,m}$.  Hence $x \in C$.  Conversely, suppose that $x \notin \{0\} \cup \bigcup_{e,m \in \Nb}I_{e,m}$.  If $x$ is not in any $I_e$ for $e \in \Nb$, then clearly $x \notin C$.  So suppose that $x \in I_e$.  Let $k$ be such that $x \in (a^e_k, b^e_k)$, and let $m$ be such that $k < k_{e,m}$.  Then $x \in \bigcup_{k \leq k_{e,m}}(a^e_k, b^e_k) \setminus \bigcup_{n \leq m}I_{e,n}$, so $x \in \bigcup_{n \leq m}\bigcup_{O \in U_{e,n}}O$.  Thus $C = \{0\} \cup \bigcup_{e,m \in \Nb}I_{e,m}$, establishing~(i).

To establish~(iv), for each $e, m \in \Nb$, transfer $\la (a_k, b_k) : k \in \Nb \ra$ to $I_{e,m}$ via the linear transformation that maps $0$ to the left endpoint of $I_{e,m}$ and maps $1$ to the right endpoint of $I_{e,m}$.  Denote the transferred sequence of intervals by $\la (a^{e,m}_k, b^{e,m}_k) : k \in \Nb \ra$.  To also ensure~(v), intersect each interval $(a^{e,m}_k, b^{e,m}_k)$ with $(c^{e,m}, d^{e,m})$, which suffices because $I_{e,m} \subseteq (c^{e,m}, d^{e,m})$, and the intervals of $\la (c^{e,m}, d^{e,m}) : m \in \Nb \ra$ are pairwise disjoint.
\end{proof}

\begin{Theorem}\label{thm-sTETWKL}
$\rca + \stet \vdash \wkl$.
\end{Theorem}

\begin{proof}
We derive the contradiction $\rca + \neg\wkl + \stet \vdash \wkl$.  Let $g_0, g_1 \colon \Nb \imp \Nb$ be injections with disjoint ranges.  Our goal is to separate the ranges of $g_0$ and $g_1$.

For each $e \in \Nb$, let $I_e = [2^{-(2e+1)}, 2^{-2e}]$.  By $\neg\wkl$, let $\la I_{e,m} : e,m \in \Nb \ra$, $\la q_{e,m} : e,m \in \Nb \ra$, and $\la (a^{e,m}_k, b^{e,m}_k) : e,m,k \in \Nb \ra$ be as in Lemma~\ref{lem-DomainHelper}, and let $D$ denote the closed set $\{0\} \cup \bigcup_{e,m \in \Nb}I_{e,m}$.  The plan is to define a continuous function $f$ with a modulus of uniform continuity on a closed and separably closed subset of $D$ such that if $F$ is a continuous extension of $f$ to $[0,1]$ with a modulus of uniform continuity, then, for each $e \in \Nb$, the value of $F(q_{e,m})$, for an $m$ chosen according to $e$ and $F$'s modulus of uniform continuity, codes whether or not $e$ should be in a separating set.

Let $E$ denote the closed set whose complement is coded by
\begin{align*}
\{(a^{e,m}_k, b^{e,m}_k) : e,m \in \Nb \andd (\forall s \leq k)(g_0(s) \neq e \andd g_1(s) \neq e)\}.
\end{align*}

Let $C$ be the closed set $D \cap E$.  Notice that for each $e,m \in \Nb$, either $I_{e,m}$ and $E$ are disjoint (if $\forall s (g_0(s) \neq e \andd g_1(s) \neq e)$) or $I_{e,m} \cap E$ is a finite union of closed intervals with rational endpoints (if $\exists s (g_0(s) = e \orr g_1(s) = e)$).  Thus $C$ is of the form $\{0\} \cup \bigcup_{e \in \Nb}J_e$ for $\la J_e : e \in \Nb \ra$ a sequence of pairwise disjoint closed intervals with rational endpoints.  Therefore $C$ is also separably closed by Lemma~\ref{lem-SepClosed}.

We now define the continuous function $f$ with modulus of uniform continuity $h$ to which we apply $\stet$.  Let
\begin{align*}
f(x) = 
\begin{cases}
0 & \text{if $x = 0$}\\
2^{-2e} & \text {if $x \in I_e \cap C \andd \exists s(g_0(s) = e)$}\\
-2^{-2e} & \text {if $x \in I_e \cap C \andd \exists s(g_1(s) = e)$}.
\end{cases}
\end{align*}

To do this, for each $e,m \in \Nb$, wait while the intervals from $\la (a^{e,m}_k, b^{e,m}_k) : k \in \Nb \ra$ covering $I_{e,m}$ are being enumerated into the complement of $C$.  If this enumeration never stops, then $I_{e,m}$ is disjoint from $C$ and $f$ is not defined on $I_{e,m}$.  If this enumeration stops at some stage $s$, then either $g_0(s) = e$ or $g_1(s) = e$, and $I_{e,m} \cap C$ is determined at this stage.  Thus the appropriate pairs of intervals can start being enumerated into the code for $f$ to define $f(x) = 2^{-2e}$ on $I_{e,m} \cap C$ if $g_0(s) = e$ and $f(x) = -2^{-2e}$ on $I_{e,m} \cap C$ if $g_1(s) = e$.

Let $h \colon \Nb \imp \Nb$ be the function $h(n) = 2n+2$.  We show that $h$ is a modulus of uniform continuity for $f$.  Suppose that $x < y$ are in $C$ and satisfy $|x-y| < 2^{-(2n+2)}$.  If $y \in I_e$ for an $e \leq n$, then $|x-y| < 2^{-(2n+2)}$ implies that $x$ must also be in $I_e$, which means that $|f(x) - f(y)| = 0 < 2^{-n}$.  If $y \in I_e$ for an $e > n$, then $|f(y)| = 2^{-2e}$ and $|f(x)| \leq 2^{-2e}$, so $|f(x) - f(y)| \leq 2^{-2e+1} < 2^{-n}$.  Thus $h$ is a modulus of uniform continuity for $f$.

By $\stet$, let $F$ be a continuous extension of $f$ to $[0,1]$ with modulus of uniform continuity $H$.  Define a set $X$ as follows.  Given $e \in \Nb$, let $m = H(2e+2)$, and use $F$ to approximate $F(q_{e,m})$ to within $2^{-(2e+2)}$ (i.e., find a rational $q$ such that $|F(q_{e,m}) - q| < 2^{-(2e+2)}$).  Define $e \in X$ if and only if this approximation is $\geq 0$.  This $X$ separates the ranges of $g_0$ and $g_1$.  Suppose $\exists s(g_0(s) = e)$.  Then $I_{e,m} \cap C \neq \emptyset$ and $F(x) = f(x) = 2^{-2e}$ on $I_{e,m} \cap C$.  As $q_{e,m} \in I_{e,m}$ and $I_{e,m}$ has length at most $2^{-m} = 2^{-H(2e+2)}$, it must be that $|F(q_{e,m})-2^{-2e}| \leq 2^{-(2e+2)}$.  Thus if $q$ approximates $F(q_{e,m})$ to within $2^{-(2e+2)}$, then $|q - 2^{-2e}| \leq 2^{-(2e+1)}$, which implies that $q$ is positive and hence that $e \in X$.  Similarly, if $\exists s(g_1(s) = e)$, then any approximation of $F(q_{e,m})$ to within $2^{-(2e+2)}$ must be within $2^{-(2e+1)}$ of $-2^{-2e}$ and thus must be negative, which implies that $e \notin X$.
\end{proof}

\begin{Corollary}
Conjecture~\ref{conj-sTET} is true.
\end{Corollary}

\begin{proof}
First, in Conjecture~\ref{conj-sTET}, (1) implies (2), (3), (4), and (5) as explained in the introduction:  $\rca$ proves the Tietze extension theorem for closed sets (i.e., the version without uniform continuity; see~\cite[Theorem~II.7.5]{SimpsonSOSOA}), and $\wkl$ proves that continuous functions on compact complete separable metric spaces have moduli of uniform continuity (see~\cite[Theorem~IV.2.2]{SimpsonSOSOA}).  Next, each of (2), (3), and (4) implies (5) because (5) is a special case of each of (2), (3), and (4).  Finally, (5) implies (1) by Theorem~\ref{thm-sTETWKL}.
\end{proof}

\section*{Acknowledgments}
We thank Steven Van den Bulcke and Hans Vernaeve for bringing Conjecture~\ref{conj-sTET} to our attention.

\bibliographystyle{amsplain}
\bibliography{TietzeExtension}

\end{document}